\documentclass[11pt,oneside,reqno]{article}
\usepackage{fix-cm}
\usepackage{authblk}
\usepackage{amssymb}
\usepackage{amsmath}
\usepackage{amsthm}
\usepackage{dsfont}
\usepackage{graphicx} 
\usepackage{varwidth}
\usepackage{geometry}

\DeclareMathOperator{\ld}{ld}
\DeclareMathOperator{\sopfr}{sopfr}

\newtheorem*{proposition}{Proposition}
\newtheorem*{question}{Question}
\theoremstyle{definition}
\newtheorem*{definition}{Definition}

\providecommand{\tabularnewline}{\\}

\newcommand{\keywords}[1]{\textbf{\textit{Keywords }} #1}

\title{A Ramanujanesque Suite}
\title{Learning from Ramanujan: Elementary Approaches to Profound Ideas}
\author{ 
Zachary P. Bradshaw$^a$ and Christophe Vignat$^b$\\
$^a$QodeX Quantum, Chicago, IL, zak@qodexquantum.ai\\
$^b$Corresponding Author, Department of Mathematics, Tulane University\\ New Orleans, LA, cvignat@tulane.edu
}
\date{ }
\begin{document}
\maketitle

\begin{abstract}
    We revisit several entries from Ramanujan’s notebooks which follow from more elementary arguments than a first glance may suggest. Our goal is to demystify these results through more accessible proofs while also shining some light on the web of interconnections within the notebooks and demonstrating the continuing relevance of Ramanujan’s methods. Classical and modern tools, such as multisection, telescoping sums, partial fraction decomposition, and Fourier analysis, are employed to reprove and extend identities originally presented without explanation. These contributions try not only to enrich our understanding of Ramanujan’s intuition but also to offer new avenues for exploration in number theory, special functions, and mathematical analysis.
\end{abstract}

\begin{quote}
    ``So I got a chance to browse through it for several weeks. It seemed quite like a revelation -- a completely new world to me, quite different from any mathematics book I had ever seen -- with much more appeal to the imagination, I must say. And frankly, it still seems very exciting to me and also retains that air of mystery which I felt at the time.'' 
    
    A. Selberg \cite{Selberg} about Ramanujan's Collected Papers book \cite{Ramanujan Collected}
\end{quote}
\section{Introduction.}

During his tragically short life, Srinivasa Ramanujan filled two notebooks with a remarkable array of mathematical discoveries, most presented without proof or explanation. These notebooks, compiled between 1903 and 1914, would go on to mystify and inspire generations of mathematicians. Decades later, Bruce Berndt undertook the monumental task of editing, proving, and contextualizing Ramanujan’s results, a tour de force that not only preserved Ramanujan’s legacy but also rendered an invaluable service to the mathematical community.

For a novice mathematician, encountering Ramanujan’s Notebooks, along with the subsequent discoveries published in the Lost Notebooks, can be a jarring experience. The sheer volume of novel and dazzling content is both exhilarating and daunting. These pages brim with originality and brilliance, but their depth is often obscured by cryptic notation, missing context, and, at times, results that seem disconnected. It is not uncommon to oscillate between awe and frustration while trying to navigate Ramanujan’s thought process.

Yet, amid this complexity, one occasionally stumbles upon a result whose elegance becomes fully apparent only after a simple, transparent argument reveals its inner structure. It is some of these moments of clarity, when Ramanujan’s genius becomes suddenly accessible, that we aim to capture and share here.

The message we wish to convey in this work is threefold:
\begin{itemize}
\item 
Alternative proofs of Ramanujan’s identities may sometimes open the door to new directions and generalizations;
\item 
Ramanujan’s notebooks are not merely a collection of isolated formulas, à la Carr’s Synopsis \cite{Carr}, but rather a complex web of interrelated ideas, making the search for hidden connections a rich and rewarding endeavor;
\item 
Many of those Ramanujan
identities which initially appear elementary are, upon closer inspection, remarkably deep.
\end{itemize} 
In what follows, we revisit a selection of entries from Ramanujan’s notebooks that, to our knowledge, have previously been explained only through elaborate or technical arguments. For each, we aim either to present a simpler proof or to offer fresh insight into the underlying structure of the result. Our hope is to inspire a renewed interest in Ramanujan’s work, especially among younger mathematicians, by showing that parts of this legendary corpus remain surprisingly accessible. Along the way, we highlight some of Ramanujan’s favorite computational tools, such as telescoping series, multisection techniques, and partial fraction decompositions, which he wielded with exceptional skill and creativity.

Throughout this paper, we use the term \emph{elementary} in a specific and modest sense. By an elementary proof, we mean one that relies only on standard tools from analysis at the undergraduate level, such as power series expansions, partial fraction decomposition, and classical identities for special functions, and that avoids contour integration, residue calculus, or ad hoc auxiliary constructions whose motivation is not transparent to beginning readers.

\section{Multisection Identities.}
Multisection of series is a technique for extracting selected coefficients from a power series $f(z) = \sum_{n=0}^{\infty} f_n z^n$,
by exploiting symmetry. Averaging $f$ over appropriate transformations of the variable $z$ isolates specific arithmetic subsequences of the coefficients $f_n$. The simplest example comes from the involution $z\mapsto-z$. Indeed, averaging $f(z)$ with $f(-z)$ gives
\begin{equation}
\frac{1}{2}\left(f(z) + f(-z)\right) = \sum_{n=0}^{\infty} f_{2n} z^{2n},
\end{equation}
which retains only the even-indexed terms of the original series. Subtracting instead extracts the odd terms of $f(z)$, producing
\begin{equation}
\frac{1}{2}\left(f(z) - f(-z)\right) = \sum_{n=0}^{\infty} f_{2n+1} z^{2n+1}.
\end{equation}

This idea extends naturally to a more general formula that allows one to extract all terms with index belonging to the residue class $q \mod p$ with $p\ge1$ and $0\le q\le p-1$. Indeed, letting $\omega = e^{\imath \frac{2 \pi }{p}}$ be the principal $p$-th root of unity, one obtains the general multisection formula \cite[Ch. 4]{Riordan}
\begin{equation}
\label{general multisection}
\frac{1}{p}\sum_{k=0}^{p-1} \frac{1}{\omega^{kq}} f\left(z\omega^k\right) = \sum_{n=0}^{\infty} f_{pn+q} z^{pn+q}.
\end{equation}

Below are a few identities that are proved in Ramanujan's notebooks \cite{R2} and \cite{R4} using the calculus of 
complex
residues. In what follows, we show that several of these identities can be recovered using multisection instead. The advantage of this viewpoint is that, once an appropriate ``base function'' is identified, the identities follow directly from known series expansions.

\subsection{Ramanujan Volume IV Entry 13 p. 380.}
Our first example appears as Entry 13 in page 380 of Volume IV of Ramanujan's Notebooks \cite{R4} in the form
\begin{equation}
\label{Entry13}
\frac{\pi}{8z^{3}}\frac{\sinh\left(2\pi z\right)+\sin\left(2\pi z\right)}{\cosh\left(2\pi z\right)-\cos\left(2\pi z\right)}=\frac{1}{8z^{4}}+\sum_{n\ge1}\frac{1}{4z^{4}+n^{4}}.
\end{equation}
This identity is proved by B. Berndt by considering it as a partial fraction decomposition and computing the associated residues. Let us take another route and prove this result using the multisection technique applied to a simpler function with a well-known expansion. To this end, define the functions
\begin{equation}
g\left(z\right)=\frac{\pi}{8z^{3}}\frac{\sinh\left(2\pi z\right)+\sin\left(2\pi z\right)}{\cosh\left(2\pi z\right)-\cos\left(2\pi z\right)},\,\,f\left(z\right)=\pi z\coth \pi z.
\end{equation}
We aim to express $g$ in terms of evaluations of $f$ at suitably rotated arguments. Letting $\omega=e^{\imath\frac{\pi}{4}}$, we observe that
\begin{equation}
f\left(z\omega\right)+f\left(z\omega^{3}\right)=\pi z\sqrt{2}\frac{\sin\left(\pi z\sqrt{2}\right)+\sinh\left(\pi z\sqrt{2}\right)}{\cosh\left(\pi z\sqrt{2}\right)-\cos\left(\pi z\sqrt{2}\right)},
\end{equation}
which follows by evaluating the classical formula \cite[4.35.37]{NIST}
\begin{equation}
\coth\left(x+\imath y\right)=\frac{\sinh2x-\imath\sin2y}{\cosh2x-\cos2y}
\end{equation}
at $x=\pi z/\sqrt{2}$ and $y=\pi z/\sqrt{2}$. Indeed, with this choice of $x,y$, the imaginary parts cancel when $f(z\omega)$ and $f(z\omega^3)$ are summed.

At this point, the connection with $g(z)$ becomes apparent. The right-hand side has the same trigonometric-hyperbolic structure as $g(z)$, differing only by a scaling of the argument and a rational prefactor. It follows that
\begin{equation}
\label{g}
g\left(z\right)=\frac{1}{16z^{4}}\left[f\left( z\omega\sqrt{2}\right)+f\left( z\omega^{3}\sqrt{2}\right)\right].
\end{equation}
This is the multisection technique in action. The function $f$ has well understood series expansions, and so identifying $g$ as a sum of parts in a multisection  of $f$ will produce a series expansion for $g$. 

To complete the derivation, we expand $f$ using its Mittag-Leffler series expansion \cite[4.36.3]{NIST} 
\begin{equation}
\label{MittagLeffler}
\pi z\coth \pi z=1+\sum_{n\ge1}\frac{2z^{2}}{z^{2}+n^{2}}.
\end{equation}
Substituting in \eqref{g} produces, after some algebra, Ramanujan's identity \eqref{Entry13}. Different choices for the base function $f(z)$ and the root of unity $\omega$ produce similar identities. The reader is invited to check the identities in Table \ref{Table1} arising from this method, and come up with more of their own.

\begin{table}[h]
\begin{tabular}{|c|c|c|c|c|}
\hline 
 & $f\left(z\right)$  & $\omega$ & multisection & identity\tabularnewline
\hline 
\hline 
(a) & $\pi z\csc\pi z$ & $e^{\imath\frac{\pi}{6}}$ & $f\left(z\omega\right)+f\left(z\omega^{5}\right)$ & 
{$\begin{aligned}
&\pi z\frac{\sqrt{3}\cosh\left(\frac{\pi z}{2}\right)\sin\left(\frac{\sqrt{3}}{2}\pi z\right)+\cos\left(\frac{\sqrt{3}}{2}\pi z\right)\sinh\left(\frac{\pi z}{2}\right)}{\cosh\pi z-\cos\left(\sqrt{3}\pi z\right)}\\
&=1+z^{2}\sum_{n\ge1}\left(-1\right)^{n}\frac{2z^{2}-n^{2}}{z^{4}-n^{2}z^{2}+n^{4}}
\end{aligned}$} \tabularnewline
\hline 
(b) & $\pi z\csc\pi z$ & $e^{\imath\frac{\pi}{6}}$ & {$\begin{aligned}
    &f\left(z\omega\right)+f\left(z\omega^{5}\right)\\
    &+\imath\sqrt{3}\\
    &\left(f\left(z\omega\right)-f\left(z\omega^{5}\right)\right)
\end{aligned}$}

& {$\begin{aligned}
    &4\pi z\frac{\cos\left(\frac{\sqrt{3}}{2}\pi z\right)\sinh\left(\frac{\pi z}{2}\right)}{\cosh\pi z-\cos\left(\sqrt{3}\pi z\right)}\\
    &=1+2z^{2}\sum_{n\ge1}\left(-1\right)^{n}\frac{z^{2}+n^{2}}{z^{4}-n^{2}z^{2}+n^{4}}
\end{aligned}$}
\tabularnewline
\hline 
(c) & $\pi^{2}\csc^{2}\pi z$ & $e^{\imath\frac{\pi}{6}}$ & $f\left(z\omega\right)+f\left(z\omega^{5}\right)$ & {$\begin{aligned}
    &4\pi^{2}\frac{1-\cos\left(\sqrt{3}\pi z\right)\cosh\left(\pi z\right)}{\left(\cosh\left(\pi z\right)-\cos\left(\sqrt{3}\pi z\right)\right)^{2}}\\
    &=\sum_{n\in\mathbb{Z}}\frac{2n^{6}-n^{4}z^{2}-5n^{2}z^{4}+z^{6}}{\left(z^{4}-n^{2}z^{2}+n^{4}\right)^{2}}
\end{aligned}$}
\tabularnewline
\hline 
(d) & $\pi z\coth\pi z$ & $e^{\imath\frac{2\pi}{3}}$ & $f\left(z\omega\right)+f\left(z\omega^{2}\right)$ & {$\begin{aligned}
    &\frac{\pi z}{2}\frac{\sqrt{3}\sin\left(\pi z\sqrt{3}\right)+\sinh\left(\pi z\right)}{\cosh\left(\pi z\right)-\cos\left(\pi z\sqrt{3}\right)}\\
    &=1+z^{2}\sum_{n\ge1}\frac{2z^{2}-n^{2}}{z^{4}-n^{2}z^{2}+n^{4}}
\end{aligned}$}
\tabularnewline
\hline 
(e) & $\pi z\coth\pi z$ & $e^{\imath\frac{2\pi}{3}}$ & $f\left(z\omega\right)-f\left(z\omega^{2}\right)$ & {$\begin{aligned}
    &\frac{\pi}{2z\sqrt{3}}\frac{\sinh\left(\pi z\sqrt{3}\right)-\sqrt{3}\sin\left(\pi z\right)}{\cosh\left(\pi z\sqrt{3}\right)-\cos\left(\pi z\right)}\\
    &=\sum_{n\ge0}\frac{n^{2}}{z^{4}+n^{2}z^{2}+n^{4}}
\end{aligned}$}
\tabularnewline
\hline 
(f) & $\pi z\coth\pi z$ & $e^{\imath\frac{\pi}{4}}$ & $f\left(z\omega\right)-f\left(z\omega^{3}\right)$ & {$\begin{aligned}
    &\frac{\pi}{4z}\frac{\sinh\left(2\pi z\right)-\sin\left(2\pi z\right)}{\cosh\left(2\pi z\right)-\cos\left(2\pi z\right)}\\
    &=\sum_{n\ge1}\frac{n^{2}}{4z^{4}+n^{4}}
\end{aligned}$}
\tabularnewline
\hline 
\end{tabular}
\caption{Generalizations of identity \eqref{Entry13}. Identities (e) and (f) can be found respectively as Entry 4 page 248 in Volume II and Entry 14 page 380 in Volume IV of Ramanujan's Notebooks.}
\label{Table1}
\end{table}

\subsection{A truly beautiful partial fraction expansion.}
Berndt describes the following identity (Entry 4 on page 360 in Notebook IV \cite{R4}) as ``a truly beautiful partial fraction expansion'':
\begin{align}
\label{truly}
\frac{\pi}{\sqrt{3}z^{2}}\frac{\sin\left(\pi z\right)\sin\left(\pi z\sqrt{3}\right)+\sinh\left(\pi z\right)\sinh\left(\pi z\sqrt{3}\right)}{\left(\cos\left(\pi z\sqrt{3}\right)-\cosh\left(\pi z\right)\right)\left(\cos\left(\pi z\right)-\cosh\left(\pi z\sqrt{3}\right)\right)}&\nonumber\\    
=\frac{1}{2\pi z^{4}}+\sum_{n\ge1}\frac{n\coth\left(n\pi\right)}{z^{4}+n^{2}z^{2}+n^{4}}+\frac{n\coth\left(n\pi\right)}{z^{4}-n^{2}z^{2}+n^{4}}.
\end{align}
Although more elaborate than \eqref{Entry13}, this identity can still be understood through multisection. The essential observation is that the rational terms on the right-hand side admit power-series expansion in $z$ with coefficients involving sums of the form $\sum_{n\ge1}\frac{\coth(n\pi)}{n^m}$. Once these sums are evaluated, the remaining steps reduce to recognizing a multisection of an explicit generating function.

We first use a result due to Nanjundiah \cite{Nanjundiah}:
\begin{equation}
\label{Nanjundiah}
\sum_{n\ge1}\frac{\coth\left(n\pi\right)}{n^{4p-1}}=\frac{1}{\pi}\sum_{k=0}^{2p}\left(-1\right)^{k-1}\zeta\left(2k\right)\zeta\left(4p-2k\right),\,\,p\ge1
\end{equation}
with $\zeta(n)$ the Riemann zeta function.  We may rewrite this identity using Euler's identity
$
\zeta\left(2n\right)=\left(-1\right)^{n-1}\frac{\left(2\pi\right)^{2n}}{2\left(2n\right)!}B_{2n},
$
as
\[
\sum_{n\ge1}\frac{\coth\left(n\pi\right)}{n^{4p-1}}=-\frac{\left(2\pi\right)^{4p}}{4\pi\left(4p\right)!}\tilde{B}_{4p},
\]
with $\tilde{B}_{n}$ the \textit{Ramanujan polynomials} defined by the convolution
$
\tilde{B}_{n}
=\sum_{k=0}^n \binom{n}{k}\imath^k B_k B_{n-k}
$
of the Bernoulli numbers $B_n$. 

On the other hand, a  partial fraction decomposition shows that
\[\frac{1}{2}\frac{n}{z^{4}+n^{2}z^{2}+n^{4}}+\frac{1}{2}\frac{n}{z^{4}-n^{2}z^{2}+n^{4}}
=\sum_{p\ge0}\frac{z^{12p}}{n^{12p+3}}-\sum_{p\ge0}\frac{z^{12p+8}}{n^{12p+11}},
\]
from which we deduce
\[
\frac{1}{2}\sum_{n\ge1}\frac{n\coth\left(n\pi\right)}{z^{4}+n^{2}z^{2}+n^{4}}+\frac{n\coth\left(n\pi\right)}{z^{4}-n^{2}z^{2}+n^{4}}
=\sum_{p\ge0}\left(\sum_{n\ge1}\frac{\coth\left(n\pi\right)}{n^{12p+3}}\right)z^{12p}-\sum_{p\ge0}\left(\sum_{n\ge1}\frac{\coth\left(n\pi\right)}{n^{12p+11}}\right)z^{12p+8}.\]
Replacing each inner sum with its Ramanujan polynomial expression and multiplying by a factor of $2$ produces
\[
-\frac{1}{4\pi z^{4}}\left[\sum_{p\ge0}\frac{\tilde{B}_{12p+4}}{\left(12p+4\right)!}\left(2\pi z\right)^{12p+4}-\sum_{p\ge0}\frac{\tilde{B}_{12p+12}}{\left(12p+12\right)!}\left(2\pi z\right)^{12p+12}\right].
\]
Up to the $2\pi$  scaling factor, we recognize the difference between two $12$-multisections of the Ramanujan polynomial generating function
\[
f\left(z\right)=\sum_{n\ge0}\frac{\tilde{B}_{n}}{n!}z^{n}=\left(\frac{z}{e^{z}-1}\right)\left(\frac{\imath z}{e^{\imath z}-1}\right).
\]
Denoting $\omega=e^{\imath\frac{\pi}{6}}$  and using \eqref{general multisection}, the multisections are computed as
\[
\sum_{p\ge0}\frac{\tilde{B}_{12p+4}}{\left(12p+4\right)!}z^{12p+4}=\frac{1}{12}\sum_{k=0}^{11}\frac{1}{\omega^{4k}}f\left(z\omega^{k}\right),\,\,\sum_{p\ge0}\frac{\tilde{B}_{12p+12}}{\left(12p+12\right)!}z^{12p+12}=\frac{1}{12}\sum_{k=0}^{11}f\left(z\omega^{k}\right)-1,
\]
and their difference is computed, after some algebra, as
\[
1-\frac{2\sqrt{3}}{12}z^{2}\frac{\sin\left(\frac{z}{2}\right)\sin\left(\frac{\sqrt{3}}{2}z\right)+\sinh\left(\frac{z}{2}\right)\sinh\left(\frac{\sqrt{3}}{2}z\right)}{\left(\cos\left(\frac{\sqrt{3}}{2}z\right)-\cosh\left(\frac{z}{2}\right)\right)\left(\cos\left(\frac{z}{2}\right)-\cosh\left(\frac{\sqrt{3}}{2}z\right)\right)}.
\]
Finally, replacing $z$ with $2\pi z$ and simplifying produces the desired result. While the algebra here was more involved than in the previous example, the underlying mechanism is the same.

\subsection{Final remark.}
It seems reasonable to assume that Ramanujan had the multisection technique in mind when he produced the previous identities. A reason why he was interested in them remains elusive, but we suggest here one potential reason. Rather than using the Mittag-Leffler series expansion of the base function $f(z),$ let us consider its Taylor series expansion.

For example, the base functions $f_1(z)=\pi z \cot \pi z$ and $f_2(x)=\pi z \csc \pi z$
have Taylor expansions
\[
f_1(z) = 1-2\sum_{k\ge1}\zeta\left(2k\right)z^{2k},\,\,
f_2(z) = \sum_{k\ge 0}\left(2-2^{2-2k}\right) \zeta (2k) z^{2k}.
\]
Ramanujan's identity \eqref{Entry13} and identities (f), (e) and (d) in Table \ref{Table1}  read equivalently
\[
\frac{\pi}{8z^{3}}\frac{\sinh\left(2\pi z\right)+\sin\left(2\pi z\right)}{\cosh\left(2\pi z\right)-\cos\left(2\pi z\right)}-\frac{1}{8z^{4}}=
\sum_{n \ge 1}(-1)^n 2^{2n} \zeta(4n) z^{4n},
\]
\[
\frac{\pi}{4z}\frac{\sinh\left(2\pi z\right)-\sin\left(2\pi z\right)}{\cosh\left(2\pi z\right)-\cos\left(2\pi z\right)} = \sum_{n \ge 0}(-1)^n 2^{2n}\zeta(4n+2) z^{4n},
\]
\[
\frac{\pi}{4z}\frac{\sinh\left(\pi z\sqrt{3}\right)-\sqrt{3}\sin\left(\pi z\right)}{\cosh\left(\pi z\sqrt{3}\right)-\cos\left(\pi z\right)}
=\sum_{n \ge 0}\cos(\frac{\pi}{6}+n\frac{2\pi}{3})\zeta(2n+2)z^{2n},
\] 
\begin{align}
\nonumber
\pi z\frac{\sqrt{3}\sin\left(\pi z\sqrt{3}\right)+\sinh\left(\pi z\right)}{\cosh\left(\pi z\right)-\cos\left(\pi z\sqrt{3}\right)}&
=-4\sum_{n\ge0}\cos\left(\frac{n\pi}{3}\right)\zeta\left(2n\right)z^{2n}
\end{align} 
with the convention $\zeta(0)=-\frac{1}{2}.$
Even the more  involved identity \eqref{truly} can be expressed as
\begin{align*}
\frac{\pi}{\sqrt{3}z^{2}}\frac{\sin\left(\pi z\right)\sin\left(\pi z\sqrt{3}\right)+\sinh\left(\pi z\right)\sinh\left(\pi z\sqrt{3}\right)}{\left(\cos\left(\pi z\sqrt{3}\right)-\cosh\left(\pi z\right)\right)\left(\cos\left(\pi z\right)-\cosh\left(\pi z\sqrt{3}\right)\right)}-\frac{1}{2\pi z^{4}}&\\
=2\sum_{p\ge0}z^{12p}\tilde{\zeta}(12p+4)-2\sum_{p\ge0}z^{12p+8}\tilde{\zeta}(12p+12)
\end{align*}
with $\tilde{\zeta}(4p)=\sum_{n\ge 1}\frac{\coth{n\pi}}{n^{4p-1}}$ the ``double  zeta'' series  given by Nanjundiah's identity
 \eqref{Nanjundiah}. These examples suggest that, \textit{in fine}, Ramanujan was looking for explicit expressions for generating  functions of the Riemann zeta function.

\section{Partial Fraction Decomposition and Telescoping.}
\subsection{Ramanujan's original identity.}
Partial fraction decomposition and telescoping are two of Ramanujan's favorite computational techniques. Let us illustrate this by considering Entry 11(iii) on page 217 of Ramanujan's  notebook II \cite{R2}, which reads
\begin{equation}\label{eq:desired-sum}
\log2\sum_{k\ge2}\frac{\left(-1\right)^{k}}{k\log k}+\log^{2}2\sum_{k\ge2}\frac{1}{k\log k\log\left(2k\right)}=1.
\end{equation}
Berndt produces a remarkable proof of this identity by considering the variant
\begin{equation}\label{eq:sn-version}
S_{n}=\sum_{k=1}^{n}\frac{1}{k\left(k+1\right)}+\sum_{k\ge2}\frac{\left(-1\right)^{k}\log2}{k\log\left(2^{n}k\right)}+\sum_{k\ge2}\frac{\log^{2}2}{k\log\left(2^{n}k\right)\log\left(2^{n+1}k\right)}
\end{equation}
and showing by induction on $n$ that $S_{n}$ coincides with the left-hand side of  \eqref{eq:desired-sum} for any value of $n$. Letting $n\to\infty$ then produces
\begin{equation}
S_{\infty}=\sum_{k\ge1}\frac{1}{k\left(k+1\right)}=1,
\end{equation}
which proves the identity. We propose here another proof that does not rely on the intuition behind Berndt's  proof, hoping that it will generate multiple generalizations of this result.

The key observation is the identity
\begin{equation}
\label{pfd log}
\frac{1}{\log k\log\left(2k\right)}=\frac{1}{\log2}\left(\frac{1}{\log k}-\frac{1}{\log\left(2k\right)}\right),
\end{equation}
which plays the same role here that a partial fraction decomposition plays for rational functions.
Ideally, we would proceed as follows: as a consequence of \eqref{pfd log}, we have
\begin{equation}
\log2\sum_{k\ge2}\frac{\left(-1\right)^{k}}{k\log k}+\log^{2}2\sum_{k\ge2}\frac{1}{k\log k\log\left(2k\right)}
=\log2\left(\sum_{k\ge2}\frac{\left(-1\right)^{k}}{k\log k}+\sum_{k\ge2}\frac{1}{k\log k}-\frac{1}{k\log\left(2k\right)}\right).
\end{equation}
Now adding like terms in the right-hand side produces
\begin{align}
\sum_{k\ge2}\frac{\left(-1\right)^{k}+1}{k\log k}-\frac{1}{k\log\left(2k\right)}&=\sum_{k\ge1}\frac{2}{2k\log\left(2k\right)}-\sum_{k\ge2}\frac{1}{k\log\left(2k\right)}
=\sum_{k\ge1}\frac{1}{k\log\left(2k\right)}-\sum_{k\ge2}\frac{1}{k\log\left(2k\right)},
\end{align}
which, by telescoping, produces the desired result $\frac{1}{\log2}.$ The key point is that the identity \eqref{pfd log} converts the problem into a telescoping sum. Once this observation is made, the argument becomes entirely transparent; all terms cancel except for a single boundary contribution. This mechanism is standard in elementary
real analysis and contrasts with Berndt’s original approach, where the guiding insight behind the auxiliary expression is less immediate.

Unfortunately, the series $\sum_{k\ge2}\frac{1}{k\log(k)}$ does not converge, and so we have to proceed with a little more caution. Instead of working with the infinite series above, we consider the partial sums, do the above manipulations, and show that the remaining tail of the sum converges to $0$. Indeed, looking at the truncated sum over the interval $\left[2,M\right]$, $M=2N$ being assumed even without loss of generality, we have 
\begin{align}
\nonumber
&\log2\left(\sum_{k=2}^{2N}\frac{\left(-1\right)^{k}}{k\log k}+\sum_{k=2}^{2N}\left(\frac{1}{k\log k}-\frac{1}{k\log\left(2k\right)}\right)\right)=\log2\left(\sum_{k=2}^{2N}\frac{\left(-1\right)^{k}+1}{k\log k}-\sum_{k=2}^{2N}\frac{1}{k\log\left(2k\right)}\right)\\
&=\log2\left(\sum_{k=1}^N\frac{1}{k\log(2k)}-\sum_{k=2}^{2N}\frac{1}{k\log(2k)}\right)
=\log2\left(\frac{1}{\log2}-\sum_{k=N+1}^{2N}\frac{1}{k\log(2k)}\right).
\end{align}
Since $k\log(2k)$ is a strictly increasing function of $k$, the remaining sum is bounded by $\frac{N}{\left(N+1\right)\log(2+2N)}$, which 
converges to $0$ as $N\to \infty$. 

In the generalizations that follow, we will make the formal manipulations of infinite series as we did above, ignoring the convergence issue,  with the understanding that they can be justified by bounding the sequence of partial sums.

\subsection{A first generalization.}

Looking back at the previous proof, it appears that the only property of $\log(k)$ that we have used is its complete additivity. This suggests that Ramanujan's identity can be extended to completely additive functions as we now show.

\begin{definition}
A function defined on the positive integers is called \textit{completely additive} if it satisfies $f(mn)=f(m)+f(n)$ for all positive integers $m,n$.     
\end{definition}

\begin{proposition}\label{prop:completely-add}
    Let $f$ be a completely additive function such that $f(k)> c\sqrt{\log(k)}$ for all $k\ge2$ and some constant $c$. Then
    \begin{equation}\label{eq:identity}
        f(2)\sum_{k\ge2}\frac{(-1)^k}{kf(k)}+(f(2))^2\sum_{k\ge2}\frac{1}{kf(k)f(2k)}=1.
    \end{equation}
\end{proposition}
\begin{proof}
    Observe that
    \begin{equation}\label{eq:pfd1}
        \frac{1}{f(2)}\left(\frac{1}{f(k)}-\frac{1}{f(2k)}\right)=\frac{1}{f(2)}\frac{f(2k)-f(k)}{f(k)f(2k)}.
    \end{equation}
    Since $f(2k) = f(k)+f(2),$ it follows that \eqref{eq:pfd1} can be written
    \begin{equation}
        \frac{1}{f(2)}\left(\frac{1}{f(k)}-\frac{1}{f(2k)}\right)=\frac{1}{f(2)}\frac{f(2)}{f(k)f(2k)}=\frac{1}{f(k)f(2k)}.
    \end{equation}
    Applying this result to the second sum in \eqref{eq:identity}, which converges by the growth assumption, produces
    \begin{align}
    \nonumber
        &f(2)\sum_{k\ge2}\frac{(-1)^k}{kf(k)}+(f(2))^2\sum_{k\ge2}\frac{1}{kf(k)f(2k)}=f(2)\left[\sum_{k\ge2}\frac{(-1)^k}{kf(k)}+\sum_{k\ge2}\frac{1}{kf(k)}-\sum_{k\ge2}\frac{1}{kf(2k)}\right]\\
        \nonumber
        &=f(2)\left[\sum_{k\ge2}\frac{(-1)^k+1}{kf(k)}-\sum_{k\ge2}\frac{1}{kf(2k)}\right]=f(2)\left[\sum_{k\ge1}\frac{2}{2kf(2k)}-\sum_{k\ge2}\frac{1}{kf(2k)}\right]=f(2)\left[\frac{1}{f(2)}\right]=1.
    \end{align}
\end{proof}

There are a number of interesting completely additive functions aside from the logarithm, including the number $\Omega(k)$ of prime divisors of $k$, the sum $\sopfr(k)$ of the prime factors of $k$, both counting multiplicity, and the logarithmic arithmetic derivative $\ld(k)$ \cite{Zachary} defined by the complete additivity property and the assumption $\ld(p)=\frac{1}{p}$ for every prime $p$. Of these, the $\sopfr$ function satisfies the growth condition needed in Proposition~\ref{prop:completely-add}, which produces the identity
\begin{equation}
        \sum_{k\ge2}\frac{(-1)^k}{k\sopfr(k)}+2\sum_{k\ge2}\frac{1}{k\sopfr(k)(\sopfr(k)+2)}=\frac12.
\end{equation}

\subsection{Another generalization.}
The same mechanism produces families of identities in which the shifts by powers of $2$ are iterated, and we record one representative statement here.

\begin{proposition}
For an arbitrary integer $q \ge 1,$ the following identity holds:
\begin{equation}
\log2 \sum_{l=1}^q
\sum_{k\ge2}\frac{\left(-1\right)^{k}}{k
\log\left(k2^l\right)}
+q\log^2(2)\sum_{k\ge2}\frac{1}{k
\log\left(2k\right)
\log\left(k2^{q+1}\right)}=\sum_{l=1}^q \frac{1}{l+1}.
\end{equation}

\end{proposition}
\begin{proof}
For $q=1,$ the partial fraction decomposition 
\begin{equation}
\nonumber
\frac{\log\left(2\right)}{\log\left(2k\right)\log\left(4k\right)}=\frac{1}{\log\left(2k\right)}-\frac{1}{\log\left(4k\right)}
\end{equation}
produces
\begin{equation}
\nonumber
\sum_{k\ge2}\frac{\left(-1\right)^{k}}{k\log\left(2k\right)}+\log2\sum_{k\ge2}\frac{1}{k\log\left(2k\right)\log\left(4k\right)}
=\sum_{k\ge2}\frac{\left(-1\right)^{k}}{k\log\left(2k\right)}+\sum_{k\ge2}\frac{1}{k\log\left(2k\right)}-
\frac{1}{k\log\left(4k\right)}.
\end{equation}
Combining the first two series produces the series $\sum_{k\ge1}
\frac{1}{k\log\left(4k\right)}$ that telescopes with the third series to  produce $\frac{1}{2\log(2)}$ so that
\begin{equation}
\nonumber
\log2\sum_{k\ge2}\frac{\left(-1\right)^{k}}{k\log\left(2k\right)}+\log^{2}2\sum_{k\ge2}\frac{1}{k\log\left(2k\right)\log\left(4k\right)}=\frac{1}{2}.
\end{equation}
Similarly, in the case $q=2,$ the partial fraction decomposition 
\begin{equation}
\nonumber
\frac{\log\left(4\right)}{\log\left(2k\right)\log\left(8k\right)}=\frac{1}{\log\left(2k\right)}-\frac{1}{\log\left(8k\right)}
\end{equation}
produces
\begin{align}
\sum_{k\ge2}&\frac{\left(-1\right)^{k}}{k\log\left(2k\right)}+\sum_{k\ge2}\frac{\left(-1\right)^{k}}{k\log\left(4k\right)}+\log4\sum_{k\ge2}\frac{1}{k\log\left(2k\right)\log\left(8k\right)}\nonumber
\\
\nonumber
&=\sum_{k\ge2}\frac{\left(-1\right)^{k}}{k\log\left(2k\right)}+\sum_{k\ge2}\frac{\left(-1\right)^{k}}{k\log\left(4k\right)}+\sum_{k\ge2}\frac{1}{k\log\left(2k\right)}-\sum_{k\ge2}\frac{1}{k\log\left(8k\right)}.
\end{align}
The first and third series combine into $\sum_{k\ge1}\frac{1}{k\log\left(4k\right)}$ which in turn combines with the second series to produce $\frac{1}{\log 4}+\sum_{k\ge1}\frac{1}{k\log\left(8k\right)}.$ This series telescopes with the last one, leaving constant terms
$\frac{\log 2}{\log 4}+\frac{\log 2}{\log 8}=\frac{1}{2}+\frac{1}{3}$ as desired.
The proof of the general case follows the same pattern and is omitted.
\end{proof}
Note that the same arguments produce the identity
\begin{equation}
\sum_{k\ge2}\frac{(-1)^k}{k \log(k2^n)} +\frac{\log2}{k\log(k2^n)\log(k2^{n+1})}=\frac{1}{(n+1)\log2},
\end{equation}
explaining why the sequence $S_n$ used by B. Berndt is constant.
\subsection{Yet another generalization.}

Observe that in our proof of identity \eqref{eq:desired-sum}, any term $\log k$ can be replaced by $\log k +z$ with an arbitrary real number $z>-\log2.$ Hence, Ramanujan's identity extends to a
parameterized deformation as follows:
\begin{proposition}
For any $z>-\log2,$
\begin{equation}
\sum_{k\ge2}\frac{\left(-1\right)^{k}}{k\left(\log k+z\right)}+\sum_{k\ge2}\frac{\log2}{k\left(\log k+z\right)\left(\log\left(2k\right)+z\right)}=\frac{1}{\log2+z}.
\end{equation}
Choosing $\alpha>\frac{1}{2}$ and $z=\log\alpha$ produces the generalization
\begin{equation}
\sum_{k\ge2}\frac{\left(-1\right)^{k}}{k\log\left(\alpha k\right)}+\sum_{k\ge2}\frac{\log2}{k\log\left(\alpha k\right)\log\left(2\alpha k\right)}=\frac{1}{\log2\alpha}.
\end{equation}
The special case $\alpha=1$ is Ramanujan's identity \eqref{eq:desired-sum}.
\end{proposition}

As in the previous results, this general identity holds for arbitrary completely additive functions satisfying a growth condition.
\begin{proposition}
    Let $\alpha\in\mathbb{Q}_{\ge1}$ and let $f$ be a completely additive function satisfying $f(x)>c\sqrt{\log(x)}$ for some constant $c$. Then
    \begin{equation}
        \sum_{k\ge2}\frac{(-1)^k}{kf(\alpha k)}+f(2)\sum_{k\ge2}\frac{1}{kf(\alpha k)f(2\alpha k)}=\frac{1}{f(2\alpha)}.
    \end{equation}
\end{proposition}
\begin{proof}
    Since $f$ is completely additive, it can be extended to $\mathbb{Q}$ by $f(m/n)=f(m)-f(n)$ and we have
    \begin{equation}
    \nonumber
        \frac{1}{f(2)}\left(\frac{1}{f(\alpha k)}-\frac{1}{f(2\alpha k)}\right)=\frac{1}{f(2)}\frac{f(2\alpha k)-f(\alpha k)}{f(\alpha k)f(2\alpha k)}=\frac{1}{f(\alpha k)f(2\alpha k)}.
    \end{equation}
    Then
    \begin{align}
    \nonumber
        &\sum_{k\ge2}\frac{(-1)^k}{kf(\alpha k)}+f(2)\sum_{k\ge2}\frac{1}{kf(\alpha k)f(2\alpha k)}=
        \sum_{k\ge2}\frac{(-1)^k}{kf(\alpha k)}+\sum_{k\ge2}\frac{1}{k}\left(\frac{1}{f(\alpha k)}-\frac{1}{f(2\alpha k)}\right)\\
        \nonumber
        &=\sum_{k\ge2}\frac{(-1)^k}{kf(\alpha k)}+\sum_{k\ge2}\frac{1}{kf(\alpha k)}-\sum_{k\ge2}\frac{1}{kf(2\alpha k)}
        =\sum_{k\ge1}\frac{1}{kf(2\alpha k)}-\sum_{k\ge2}\frac{1}{kf(2\alpha k)}=\frac{1}{f(2\alpha)},
    \end{align}
    where the growth condition has been used to guarantee the convergence of the sums.
\end{proof}

\section{A trigonometric integral.}
\begin{quote}
    ``Of Ramanujan’s remarkable ability to evaluate elliptic and other definite integrals, Hardy has said that during the course of his lectures, if at any time he needed the value of a certain integral, he would simply turn towards Ramanujan in the audience, who would provide the answer instantly!'' \cite{Alladi}, p.117
\end{quote}
Entry 16(ii) in page 264 of Ramanujan's Notebook II  \cite{R2} reads: 
\begin{proposition}
For $n,p$ integers
\begin{equation}
\label{Inp}
I_{n,p}=\int_{0}^{\infty}\frac{\sin^{2n+1}x}{x}\cos\left(2px\right)dx
=\left(-1\right)^{p}\frac{\sqrt{\pi}}{2}\frac{\Gamma\left(n+1\right)\Gamma\left(n+\frac{1}{2}\right)}{\Gamma\left(n-p+1\right)\Gamma\left(n+p+1\right)}.
\end{equation}
Moreover,
\begin{equation}
\label{Inp2}
I_{n,p}=\int_{0}^{\infty}\frac{\sin^{2n+2}x}{x^{2}}\cos\left(2px\right)dx.
\end{equation}

\end{proposition}
The proof in \cite{R2} is obtained as a consequence of the recurrence identity 
\begin{equation}
I_{n,p}=\frac{1}{2}I_{n-1,p}-\frac{1}{4}I_{n-1,p+1}-\frac{1}{4}I_{n-1,p-1}.
\end{equation}
We propose here two different proofs for this result, both proofs producing the value of the integral $I_{n,p}$ for any integer $n$ and \textit{any real number} $p.$ Obviously, this value coincides with Ramanujan's formula \eqref{Inp} at integer values of $p.$
The first proof is based on a lesser-known identity by Cauchy. The second proof uses standard tools of Fourier analysis, and would be the one produced by a current graduate student.
The interested  reader will find in  \cite{Russel} yet another proof based on the more exotic technique of \textit{integration by differentiation}.

\subsection{An identity by Cauchy.}
Our guess is that Ramanujan did not use Fourier analysis\footnote{See the  Conclusion section for further comment.} to compute the integral $I_{n,p}$.  The fact that he computed this  integral for integer values of the parameter $p$ only, strongly suggests that he used a recursive method, like in B. Berndt's proof. 
Why did he not, rather, use one of the following identities due to Cauchy, that appear as Entries 2723 and 2719 respectively  in Carr's Synopsis \cite{Carr}? This remains to us a mystery, all the more since Carr's Synopsis is known as one of the few books that strongly influenced Ramanujan.  
\begin{proposition}
With $\Delta_p$ the forward difference operator acting as $\Delta_p f(p)=f(p+1)-f(p),$ the following identities hold for all integers $n$ and real numbers $p:$
\begin{equation}
\Delta_p^{2n} \cos \left(
\left(2p-2n
\right)x\right) = (-1)^n 2^{2n} \cos\left(2px\right)\sin^{2n}\left(x\right)    
\end{equation}
and
\begin{equation}
\label{Cauchy2}
\Delta_p^{2n+1} \sin \left(
\left(2p-2n-1
\right)x\right) = (-1)^n 2^{2n+1} \cos(2px) \sin^{2n+1}\left(x\right).
\end{equation}
\end{proposition}
The second identity produces indeed a straightforward computation of the integral $I_{n,p}$:
applying rule \eqref{Cauchy2} to Ramanujan's  identity \eqref{Inp} and using the linearity property of the operator $\Delta_p$ produces
\begin{equation}
I_{n,p}=\frac{(-1)^n}{ 2^{2n+1}}
\Delta_p^{2n+1} \int_{0}^{\infty}
\frac{\sin \left(
\left(2p-2n-1
\right)x\right)}{x}dx. 
\end{equation}
The integral is easily computed as $\frac{\pi}{2}  \text{sign} (2p-2n-1)$ so that the action of $\Delta_p$ on this integral is equal to the indicator function $\pi  \mathds{1}_{\left[n-\frac{1}{2},n+\frac{1}{2}\right]}(p)$ and
\begin{equation}
I_{n,p} = \pi\frac{(-1)^n}{ 2^{2n+1}}
\Delta_p^{2n}   \mathds{1}_{\left[n-\frac{1}{2},n+\frac{1}{2}\right]}(p).
\end{equation}
It is an easy exercise left to the reader to check that, for all integer values $n \ge 0,$ this formula for $I_{n,p}$ coincides with formula \eqref{InpFourier} for all real values of $p,$ and with Ramanujan's formula \eqref{Inp2} for all integer values of $p.$

\subsection{Fourier analysis.}
The second proof uses elementary techniques from Fourier analysis. Recall the Fourier transform $FT[f]$ of a function $f$ as the function 
\begin{equation}
FT \left[ f\right](\nu) =  \int_{-\infty}^{\infty}f(x) e^{-\imath2\pi \nu x}dx.
\end{equation}
The parity of the function $\frac{\sin^{2n+1}x}{x}$ allows us to identify the integral $I_{n,p}$ in \eqref{Inp} as the Fourier transform
\begin{equation}
I_{n,p}=
\frac{1}{2}\int_{-\infty}^{\infty}\frac{\sin^{2n+1}x}{x}e^{\imath2px}dx
=\frac{1}{2}FT\left[\frac{\sin^{2n+1}x}{x}\right](\nu)
\end{equation}
with $\nu = \frac{p}{\pi}.$ 
Next, express this integral as 
\begin{equation}
    \int_{-\infty}^{\infty}\frac{\sin x}{x}\sin^{2n}xe^{\imath2px}dx,
\end{equation}
and substitute in the integral 
the binomial expansion
\begin{equation}
\sin^{2n}\left(x\right)
=\left(\frac{e^{\imath x}-e^{-\imath x}}{2\imath}\right)^{2n}
=
\frac{1}{\left(2\imath\right)^{2n}}\sum_{k=0}^{2n}\binom{2n}{k}\left(-1\right)^{k}e^{\imath\left(2k-2n\right)x}
\end{equation}
to obtain
\begin{equation}
 \frac{1}{\left(2\imath\right)^{2n}}\sum_{k=0}^{2n}\binom{2n}{k}
\left(-1\right)^{k}
\int_{-\infty}^{\infty}\frac{\sin x}{x}
 e^{\imath\left(2k-2n+2p\right)x}dx.
\end{equation}
Using the fact that the Fourier transform of the  cardinal sine  function is 
\begin{equation}
FT\left[\frac{\sin x}{x}\right]=\mathds{1}_{\left[-\frac{1}{2\pi},\frac{1}{2\pi}\right]}\left(\nu\right),
\end{equation}
the indicator function of the interval $\left[-\frac{1}{2\pi},\frac{1}{2\pi}\right],$ we deduce after simplification 

\begin{equation}
\label{InpFourier}
I_{n,p}=\pi \frac{\left(-1\right)^{n}}{2^{2n+1}}\sum_{k=0}^{2n}\binom{2n}{k}
\left(-1\right)^{k}
\mathds{1}_{\left[-\frac{1}{2},\frac{1}{2}\right]}(p+k-n).
\end{equation}

This is an even, piecewise constant function of the real variable $p$. Its values at integer argument $p$ are easily shown to coincide with Ramanujan's result \eqref{Inp}. For example, when $n=1,$
\begin{align}
I_{1,p}=
\begin{cases}
0, & \vert p\vert \ge\frac{3}{2}\\
-\frac{\pi}{8}, & \frac{1}{2}\le \vert p\vert <\frac{3}{2}\\
\frac{\pi}{4}, & 0\le \vert p\vert <\frac{1}{2}.
\end{cases}
\end{align}
The values $I_{1,0}=\frac{\pi}{4},$ $I_{1,\pm 1}=-\frac{\pi}{8}$ and $I_{1,p}=0$ for integer $p\in \mathbb{Z} \setminus \left\{-1,0,1\right\}$ coincide with those given by Ramanujan's formula \eqref{Inp}.

Using the same arguments in the case of integral \eqref{Inp2} and the fact that 
\begin{equation}
FT\left[\frac{\sin^2 x}{x^2}\right]=\mathds{T}_{\left[-\frac{1}{\pi},\frac{1}{\pi}\right]}\left(\nu\right),
\end{equation}
where the triangular function over the interval $\left[-\frac{1}{\pi},\frac{1}{\pi}\right]$ is defined by
\[
\mathds{T}_{\left[-\frac{1}{\pi},\frac{1}{\pi}\right]}(\nu) = \begin{cases}
    1-\pi\vert \nu \vert, & \vert \nu \vert < \frac{1}{\pi}  \\
    0, & \text{else},
\end{cases}
\]
we obtain the expansion
\[
\int_{0}^{\infty}\frac{\sin^{2n+2}x}{x^{2}}\cos\left(2px\right)dx\\
=\pi \frac{\left(-1\right)^{n}}{2^{2n+1}}\sum_{k=0}^{2n}\binom{2n}{k}
\left(-1\right)^{k}
\mathds{T}_{\left[-1,1\right]}(p+k-n),
\]
a piecewise linear function
that coincides with $I_{n,p}$ as given in \eqref{InpFourier} in the case where $p$ is an integer.

\section{Conclusion.}
It is our hope that this modest collection of examples will spark in young mathematicians a desire to explore Ramanujan’s universe for themselves. The remarkable breadth of Ramanujan’s work, spanning subjects from number theory to special functions, makes it easy for the interested reader to begin their journey from whatever domain they find most appealing.

Beside the five notebooks and the five lost notebooks, Ramanujan's collected papers \cite{Ramanujan Collected} represent an alternate entry point to his  monumental work. In addition to all papers published by Ramanujan, they include the solutions that he wrote to several problems published in the {\it  Journal of the Indian Mathematical Society}. Since all involve proofs, they give a glimpse at Ramanujan's way of thinking. 

Here is an example:  Question 295 \cite[p.324]{Ramanujan Collected}  in the {\it Journal of the Indian Mathematical Society} is
\begin{question}
    If $\alpha \beta = \pi,$ show that
    \[
    \sqrt{\alpha} \int_{0}^{\infty} \frac{e^{-x^2}dx}{\cosh \alpha x} = \sqrt{\beta} \int_{0}^{\infty} \frac{e^{-x^2}dx}{\cosh \beta x}.
    \]
\end{question} 
Ramanujan's solution to this question seems to indicate that he was not familiar with  Plancherel-Parseval's identity
\[
\int_{-\infty}^{\infty}f(x)\bar{g}(x)dx = \int_{-\infty}^{\infty}\tilde{f}(\nu)\overline{\tilde{g}}(\nu)d\nu
\]
(with $\tilde{f}=FT\left[f\right],$ $\tilde{g}=FT\left[g\right]$ and $\bar{g}$ the complex conjugate of $g$),
which would be the most appropriate tool to answer this question. However, Section 4 in \cite[p.55]{Ramanujan Collected} shows that Ramanujan had, in fact, created his own  and more general version of it:
\begin{proposition}
    Let
    \[
    \int_{a}^{b} f(x)F(nx)dx = \psi(n)
    \]
    and
    \[
    \int_{\alpha}^{\beta} \phi(x)F(nx)dx = \chi(n).
    \]
    Then, if we suppose the functions $f$, $\phi$, and $F$ to be such that the order of integration is indifferent, we have
    \[
    \int_{a}^{b} f(x)\chi(nx)dx =\int_{\alpha}^{\beta} \phi(y)\psi(ny)dy.
    \]
\end{proposition}
The Plancherel-Parseval identity corresponds to the special case $F(x)=e^{-\imath 2 \pi x},\,\,\alpha=a=-\infty,\,\,\beta=b=+\infty.$ In a remark that characterizes accurately his sense for discovery, Ramanujan adds: 
``A number of curious relations between definite integrals may be deduced from this result.''
More insight about Ramanujan's use of the Fourier method can be found in the dedicated Chapter 13 of \cite{R4Lost}.\\

\end{document}